\documentclass[11pt]{amsart}	
\usepackage[utf8]{inputenc}

\usepackage{graphicx}
\usepackage{subcaption}
\usepackage[usenames,dvipsnames]{pstricks}
\usepackage{epsfig}
\usepackage{pst-grad} 
\usepackage{pst-plot} 
\usepackage[space]{grffile} 
\usepackage{etoolbox} 
\makeatletter 
\patchcmd\Gread@eps{\@inputcheck#1 }{\@inputcheck"#1"\relax}{}{}
\makeatother
\usepackage{comment}
\usepackage{mathtools}
\usepackage{amsmath}
\usepackage{amssymb}
\usepackage{amsthm}
\usepackage{ifthen}
\usepackage{color}
\usepackage[UKenglish]{babel}
\usepackage{polski}

\usepackage[shortlabels]{enumitem}

\numberwithin{equation}{section}

\newcommand{\intav}[1]{\mathchoice {\mathop{\vrule width 6pt height 3 pt depth  -2.5pt
\kern -8pt \intop}\nolimits_{\kern -6pt#1}} {\mathop{\vrule width
5pt height 3  pt depth -2.6pt \kern -6pt \intop}\nolimits_{#1}}
{\mathop{\vrule width 5pt height 3 pt depth -2.6pt \kern -6pt
\intop}\nolimits_{#1}} {\mathop{\vrule width 5pt height 3 pt depth
-2.6pt \kern -6pt \intop}\nolimits_{#1}}}

\def\polhk#1{\setbox0=\hbox{#1}{\ooalign{\hidewidth\lower1.5ex\hbox{`}\hidewidth\crcr\unhbox0}}}

\newcommand{\supp}{\operatorname{supp}}

\renewcommand{\div}{\operatorname{div}}

\newcommand{\osc}{\operatorname{osc}}

 \newcommand{\Rr}{\mathbb R}

\newcommand{\Nn}{\mathbb N}

\renewcommand{\div}{\operatorname{div}}

\newcommand{\dist}{\operatorname{dist}}

\DeclareMathOperator*{\esslimsup}{ess\,lim\,sup}
\DeclareMathOperator*{\essliminf}{ess\,lim\,inf}

\newcommand{\ssubset}{\subset\joinrel\subset}

\newtheorem{Theorem}{Theorem}
\newtheorem{Definition}{Definition}
\newtheorem{Lemma}{Lemma}
\newtheorem{Corollary}{Corollary}
\newtheorem{Proposition}{Proposition}

\newtheorem{assump}{}

\newtheoremstyle{break}
  {\topsep}{\topsep}%
  {\itshape}{}%
  {\bfseries}{}%
  {\newline}{}%
\theoremstyle{break}

\theoremstyle{definition}

\theoremstyle{remark}
\newtheorem{Remark}{Remark}

\usepackage{setspace}
\title[Fully nonlinear equations with degeneracy]{Gradient regularity for fully nonlinear equations with degenerate coefficients}

\author[D. Jesus]{David Jesus}
\address{Dipartimento di Matematica, Universit\`a di Bologna, 
Piazza di Porta San Donato 5, 40126 Bologna, Italy}{}
\email{david.jesus2@unibo.it}

\author[Y. Sire]{Yannick Sire}
\address{Department of Mathematics,
Johns Hopkins University,
3400 N. Charles Street, Baltimore, MD 21218, U.S.A.}{}
\email{ysire1@jhu.edu}

\date{}

\begin{document}
\begin{abstract}
We derive $C^{1,\alpha}$ estimates for viscosity solutions of fully nonlinear equations degenerating on a hypersurface. 
\end{abstract}

\maketitle

\tableofcontents

\section{Introduction and main results}

\subsection{Introduction}The purpose of this paper is to understand how the degeneracy/singularity of the coefficients influences the regularity of the solution of fully nonlinear equations of Hessian  type. Weights degenerating as a distance to a set generalize the well-known Muckenhoupt weights, which have inumerous important applications in harmonic analysis, partial differential equations, and related areas. In particular, elliptic equations degenerating as a Muckenhoupt weight have been extensively considered, see for example \cite{CafSil, Dong-Kim_2015, Fabes-Jerison-Kenig_1982, Fabes-Kenig-Serapioni_1982, Krylov_1994, Krylov_1999, Sire-Terracini-Vita_2020, Sire-Terracini-Vita_2021}. Of particular interest is the equation in divergence form
\begin{align}\label{eq_VS}
    \mathcal{L}_au:=\div\left(|y|^aA(x,y)Du\right)=f
\end{align}
which has a close relation to the fractional Laplacian $(-\Delta)^s$ for $0<s<1$ by the famous paper of Caffarelli and Silvestre \cite{CafSil}. In a series of papers with Terracini, Tortone and Vita, the second author investigated thoroughly \eqref{eq_VS} both at the (higher) regularity level  \cite{Sire-Terracini-Vita_2020, Sire-Terracini-Vita_2021} and geometric properties of solutions \cite{STT2020}. Notice that the homogeneous weight $|y|^a$ is very special. In particular, in the aforementioned papers, the authors consider super-degenerate or super-singular weights, i.e. beyond the $A_2$ class with $a<-1$ or $a>1$. In these special cases, the associated measure is of course doubling but the weight is lacking several properties which are of fundamental use such as boundedness of maximal functions for instance. 

As far as regularity is concerned, when the domains are upper-half spaces, more general results on the existence and regularity estimates in weighted and mixed-norm Sobolev spaces for a similar class of linear elliptic and parabolic equations can be found first in \cite{DP19} with $\alpha \in (-1, 1)$ and then in \cite{DP20} with $\alpha \in (-1, \infty)$.  Similar results for problems with homogeneous Dirichlet boundary conditions can be found in \cite{DP-2023, DP-2021}. See also  \cite{CMP} for results on $W^{1,p}$-estimates for solutions of linear elliptic equations whose coefficients can be singular or degenerate with general $A_2$-weights but with some restrictive smallness assumption on the weighted mean oscillations of the coefficients. In the above-mentioned papers, some non-divergence equations have been also considered.

We would like also to point out that such coefficients involving suitable powers of the distance function to a smooth lower dimensional set have been considered by David, Feneuil and Mayboroda in their investigations about elliptic measure (for second order operators) in higher co-dimension (see e.g. \cite{DFM1,DFM2,DFM3} and references therein together with subsequent works). One of the motivations of the problem under consideration here is also to develop regularity tools for the non-divergence case in this setup. 

In the present paper we are interested in the fully nonlinear case only for Hessian equations whose {\sl coefficients} have some degeneracy.  It is rather clear by some explicit examples, that one cannot hope for higher regularity at the gradient level without an {\sl a priori} structural assumption on the degeneracy in terms of closeness to a model equation. Though, we do not claim full generality, our assumptions are motivated by recent results correlated to either nonlocal operators or higher co-dimensional harmonic analysis.  

More concretely, let $\Omega\subset \Rr^d$ be an open bounded domain; we investigate viscosity solutions to the following fully nonlinear elliptic Hessian equation
\begin{align}\label{eq_var_coef_main_form}
    F(D^2u,x)=f, \mbox{ in } \Omega,
\end{align}
where $F$ is elliptic in its first argument but {\sl degenerate} in its second, i.e. the dependence in the variable $x \in \Omega$ is allowed to vanish to some order. See assumption \ref{Assumption1} below for a more precise statement. As explained before, several recent results have been dealing with such setup both in the divergence case and the nondivergence case. Our primary goal here is to investigate the fully nonlinear theory for some natural class of equations but also provide a general framework for a subsequent work of the first author on transmission problems.

\subsection{Main results.} 

We start by introducing the main assumptions.

\begin{assump}[Degenerate ellipticity]\label{Assumption1}
For  constants $0<\lambda\leq \Lambda<\infty$, let $\Gamma$ be a smooth hypersurface which divides $\Omega$ into two connected components and define $\omega(x)=\mbox{dist}(x,\Gamma)^a$ for $a>0$. Then, for every symmetric matrices $M, N\in \text{Sym}(d)$ with $N\geq 0$, 
\[
\lambda\, \omega(x)\,|N|\leq F(M+N,x)-F(M,x)\leq \Lambda\, \omega(x)\, |N|.
\]
\end{assump}

\begin{assump}[Continuity and integrability]\label{Assumption2}
    Assume that $F(M,\cdot)\in C(\overline{\Omega})$ for every symmetric matrix $M\in \text{Sym}(d)$,  $f\in C(\overline{\Omega})\cap L^\infty(\Omega)$ and $f\omega^{-1}\in L^d(\Omega)$. 
\end{assump}

\begin{assump}[Oscillation control]\label{Assumption4}
Let $x_0\in \Omega$ and define the oscillation of $F$ at the point $x_0$ by
\begin{align*}
\beta_{F,x_0}(x):=\sup_{M\in \text{Sym}(d)\setminus \{0\}}\frac{|F(M,x)-F(M,x_0)|}{|M|}.
\end{align*}
Then for any $0<r<1$,
\begin{align*}
r^{-1}\left(\int_{B_r(x_0) \cap \Omega}\beta_{F,x_0}^d(x)\,dx \right)^{1/d}=o(r^a).
\end{align*}
\end{assump}

Our main regularity result is the following regularity statement. 
\begin{Theorem}[Interior H\"older differentiability]\label{Theorem_C1alpha}
    Suppose \ref{Assumption1}-\ref{Assumption4} are in force and let $u\in C(\Omega)$ be a viscosity solution to
    \begin{align*}
        F(D^2u,x)=f, \mbox{ in } \Omega.
    \end{align*}
    Then $u\in C_{loc}^{1,\alpha}(\Omega)$ and for every $K\ssubset \Omega$
    \begin{align*}
        \|u\|_{C^{1,\alpha}(K)}\leq C\left(\|u\|_{L^\infty(\Omega)}+\|f\|_{L^\infty(\Omega)}\right),
    \end{align*}
	where $\alpha=\min\{\alpha_0^-,1-a\}$ and $\alpha_0, C$ are universal constants depending only on $d, \lambda, \Lambda$ and $C$ depends additionally on the distance from $K$ to $\partial \Omega$.

\end{Theorem}

Here and in the rest of the paper, the constant $\alpha_0$ always refers to the universal exponent (see \cite[Corollary 5.7]{Caffarelli-Cabre_1995}) corresponding to the interior $C^{1,\alpha_0}$ regularity of solutions to the homogeneous equation 
\[
F(D^2u)=0.
\]
Furthermore, the expression $\alpha=\min\{\alpha_0^-,1-a\}$ should be understood in the following sense: if $\alpha_0>1-a$, then solutions are locally $C^{1,\alpha}$ for $\alpha=1-a$. On the other hand, if $\alpha_0\leq 1-a$, then solutions are locally $C^{1,\alpha}$ for every $\alpha<\alpha_0$.

\begin{Remark}
We would like to remark that elliptic and parabolic equations involving coefficients depending on the distance to some submanifold appear very natural in the analysis of singular manifolds with conic or conic-edge singularities.

\end{Remark}

\subsection{Comments on our approach} In this section, we start by describing the roadmap to our main result and how it differs from the one in the uniformly elliptic case, then we comment on the type of solutions considered.

We aim to follow the arguments presented in \cite{Caffarelli-Cabre_1995} within the context of uniformly elliptic equations. To begin, we want to normalize the equation by dividing by the degeneracy $\omega$ to transform it into a uniformly elliptic operator. However, this straightforward approach encounters an obstacle: since $\omega$ equals zero on $\Gamma$, this normalization isn't always feasible. As illustrated by the counterexample below, dividing the equation by $\omega$ is permissible only within the framework of $L^p$ viscosity solutions (defined below). The reason being, the broader class of $C$ viscosity solutions lacks uniqueness, thereby rendering division by $\omega$ untenable. Consequently, a crucial part of this paper is establishing the uniqueness of $L^p$ viscosity solutions for the homogeneous equation, as delineated in Theorem \ref{Theorem_uniqueness}. With uniqueness established, we can derive $C^{1,\alpha}$ regularity for the homogeneous equation in the form $\omega(x)\bar F(D^2u)=0$. In order to lay the groundwork for employing a Caffarelli-style argument, we observe that at sufficiently small scales near $\Gamma$, our equation is very close to this one. This enables us to conclude that solutions to the non-homogeneous equation are also $C^{1,\alpha}$.

Now we will argue why, even though all the data is assumed to be continuous, it turns out that the correct framework to study this type of degenerate equations is the notion of $L^p$ viscosity solution. This fact can be easily illustrated by the following example in one dimension. One can check that, for any $a,b>0$,
\[
u(t)=at_++bt_-
\]
is a $C$ viscosity solution to
\[
|t|^{\beta}u''(t)=0, \quad t\in (-1,1),
\]
for $\beta\in (0,1)$. This simple example proves that in the framework of $C$ viscosity solution, many essential properties are lost, such as uniqueness and maximum principle. However when we consider the more restrictive class of $L^p$ viscosity solutions, we can test the equation against the $L^p$ strong solution 
\[
	\phi(t)=|t|^{2-\beta}\in W_{loc}^{2,p}(-1,1)
\]
for which $u-\phi$ has a local minimum at $t=0$; however
\[
	\essliminf_{t\to 0} |t|^{\beta}\varphi''(t)=(2-\beta)(1-\beta)>0
\]
and therefore the supersolution condition fails. This peculiar behavior is justified by the following argument. If we consider the equation
\[
	F(D^2u,x)=f(x) \quad\mbox{in} B_1,
\]
where $F$ satisfies \ref{Assumption1}, then formally we can divide the equation by the degeneracy $\omega(x)$ and obtain a new equation which is uniformly elliptic but has merely measurable coefficients and source term. Therefore, as pointed out in \cite[Example 2.4]{Caffarelli-Crandall-Kocan-Swiech_1996}, the issue of uniqueness of $C$ viscosity solutions for this type of equations fails drastically. However, we note that since we assume that all the data is continuous, every $L^p$ viscosity solution is also a $C$ viscosity solution, so we can often avoid working with the technical aspects of this definition.

Now we comment on possible extensions on our result. The notion of $L^p$ viscosity solution was essential for the uniqueness of solutions, therefore if we wished to work with $C$ viscosity solutions, then one would have to impose a condition across $\Gamma$ connecting the two components of $\Omega$, which would turn our problem into a fixed transmission problem, as for example in \cite{Caffarelli-Carro-Stinga_2021}, \cite{Silva-Ferrari-Salsa_2018}, albeit governed by an operator which degenerates as we approach the interface.

Finally, we would like to point out that if we assume $\Omega=B_1 $, $x=(x',x_d)$, $\Gamma=\{x_d=0\}$ and assume that the equation degenerates with variable exponent namely $\omega(x)=|x_d|^{a(x')}$, then following the argument developed in \cite{Jesus_2022}, at each interior point $x_0=(x_0',0)$ we could obtain the sharp pointwise $C^{1,\alpha(x_0)}$ regularity for $\alpha(x_0)=\min\{1-a(x_0'),\alpha_0^-\}$, provided $a(\cdot)$ has a suitable modulus of continuity.

\section{Notations, standard definitions and basic properties}\label{Section_notation}

We now set some notation and definitions to be used throughout the paper. For completeness we give the definition of $C$ viscosity solution.

\begin{Definition}[$C$ viscosity solution]
We say $u\in C(B_1)$ is a $C$ viscosity subsolution (resp. supersolution) of
\begin{align*}
F(D^2u,x)=f, \quad \mbox{ in } B_1,
\end{align*}
if, for every $\varphi \in C^2(B_1)$ such that $u-\varphi$ has a local maximum (resp. minimum) at $x_0\in B_1$, one has
\begin{align*}
&F(D^2\varphi(x_0),x_0)\geq f(x_0)\\
\mbox{(resp.}\, & F(D^2\varphi(x_0),x_0)\leq f(x_0) \mbox{)}  .
\end{align*}
We say $u$ is an $C$ viscosity solution if it is both a subsolution and a supersolution.
\end{Definition}

As explained before,  we need as well the notion of $L^p$ viscosity solutions.

\begin{Definition}[$L^p$ viscosity solution]
We say $u\in C(B_1)$ is an $L^p$ viscosity subsolution (resp. supersolution) of
\begin{align*}
F(D^2u,x)=f, \quad \mbox{ in } B_1,
\end{align*}
if, for every $\varphi \in W^{2,p}_{loc}(B_1)$ such that $u-\varphi$ has a local maximum (resp. minimum) at $x_0\in B_1$, one has
\begin{align*}
&\esslimsup_{x\to x_0}\left(F(D^2\varphi(x),x)-f(x)\right)\geq 0\\
\mbox{(resp.} &\essliminf_{x\to x_0}\left(F(D^2\varphi(x),x)-f(x)\right)\leq 0 \mbox{).}
\end{align*}
We say $u$ is an $L^p$ viscosity solution if it is both a subsolution and a supersolution.
\end{Definition}

We also define the space $C^{1,\alpha}$ using the characterization by Campanato. Notice in the following definition the differences between in the {\sl pointwise} expansion at a given point and the validity of this on a whole neighbourhood. 

\begin{Definition}
We say $u$ is pointwise $C^{1,\alpha}$ at $x$, and write $u\in C^{1,\alpha}(x)$, if there exists an affine function $\ell_x(y)=a_x+b_x\cdot(y-x)$ and a positive constant $K_x$, possibly depending on $x$, such that
\begin{align*}
    ||u-\ell_x||_{L^\infty(B_r(x))}\leq K_xr^{1+\alpha}.
\end{align*}
If $u\in C^{1,\alpha}(x)$ for every $x\in \Omega$ and $|a_x|+|b_x|+4K_x\leq M$, then $u\in C^{1,\alpha}(\Omega)$ in the usual sense and
\begin{align*}
    ||u||_{C^{1,\alpha}(\Omega)}\leq M.
\end{align*}

\end{Definition}

 It will be useful to consider the following extremal inequalities in the $L^p$ viscosity sense, namely
\begin{align}\label{equation_Sw_1}
    \omega(x)\mathcal{M}^-(D^2u)\leq f,\quad \mbox{ in } B_R
\end{align}
and
\begin{align}\label{equation_Sw_2}
    \omega(x)\mathcal{M}^+(D^2u)\geq f,\quad \mbox{ in } B_R,
\end{align}
where $\mathcal{M}^\pm$ are the usual Pucci operators. 

Finally, we say that $u\in \overline{S}(f)$ if it satisfies \eqref{equation_Sw_1} and $u\in \underline{S}(f)$ if it satisfies \eqref{equation_Sw_2}. We also define  $S(f)=\overline{S}(f)\cap \underline{S}(f)$ and $S^*(f)=\overline{S}(|f|)\cap \underline{S}(-|f|)$.

\subsection*{H\"older Regularity of viscosity solutions}\label{Section_Calpha}. As alluded in the introduction, the zero order regularity can be derived in a straightforward way invoking known results. This amounts to divide the equation by the weight $\omega$. We still give the proof for sake of completeness, while emphasizing that these results can be obtained for more general weights. We start with the following elementary lemma.

\begin{Proposition}\label{Proposition_w_weighted_ABP}
Let $u\in \overline{S}(f)$ satisfy $u\geq 0$ in $\partial B_R$. Then
\begin{align*}
    \sup_{B_R}u^-\leq C(d)R\left(\int_{B_R\cap \{u=\Gamma_u\}}(f^+)^d\omega^{-d}\, dx\right)^\frac{1}{d}.
\end{align*}
Note that $\Gamma_u$ corresponds to the convex envelope of $u^-$ extended by $0$ to $B_{2R}$, just as in \cite{Caffarelli-Cabre_1995}.
\end{Proposition}

The proof of this result will be ommited since it follows easily by adapting the argument in \cite[Proposition 3.3]{Caffarelli-Crandall-Kocan-Swiech_1996}.  Next we construct a barrier.

\begin{Lemma}\label{Lemma_barrier}
    There exists a smooth function $\varphi$ satisfying
    \begin{align}
        \label{eq1lem8}&\varphi\geq 0 \mbox{ in } \Rr^d\setminus B_{2\sqrt{d}},\\
        \label{eq2lem8}&\varphi\leq -2 \mbox{ in } Q_3,\\
        \label{eq3lem8}&\omega(x)\mathcal{M}^+(D^2\varphi)\leq C\xi,
    \end{align}
    where $0\leq\xi\leq 1$, $\xi\in C(\Rr^d)$, $\supp\xi\subset\overline{Q}_1$. Also $\varphi\geq -M$ in $\Rr^d$.
\end{Lemma}
\begin{proof}
Let $\varphi(x)=M_1-M_2|x|^{-\alpha}$ in $\Rr^d\setminus B_{1/4}$ with $\alpha=\max\{1,(d-1)\Lambda/\lambda-1\}$ and take $M_1,M_2$ such that
\begin{align*}
    \varphi=0 \mbox{ in } B_{2\sqrt{d}}\quad \mbox{and}\quad  \varphi=-2 \mbox{ in } B_{3\sqrt{d}/2},
\end{align*}
so that \eqref{eq1lem8} holds. We can extend $\varphi$ smoothly to $\Rr^d$ so that \eqref{eq2lem8} holds. This extension depends only on $d, \lambda, \Lambda$. We now check \eqref{eq3lem8}. Note that for $|x|>1/4$
\begin{align*}
    D^2\varphi(x)=M_2\alpha|x|^{-\alpha-2}\left( -(\alpha+2)\frac{x\otimes x}{|x|^2} +I \right)
\end{align*}
and so the eigenvalues are
\begin{align*}
    M_2\alpha|x|^{-\alpha-2}(-\alpha-1,1,\hdots,1).
\end{align*}
We can calculate for $|x|>1/4$,
\begin{align*}
    \omega(x)\mathcal{M}^+(D^2\varphi)=  \omega(x) M_2\alpha|x|^{-\alpha-2}\left( (d-1)\Lambda-\lambda(\alpha+1) \right)\leq 0
\end{align*}
by our choice of $\alpha$. Since $\varphi$ was extended smoothly to $\Rr^d$ in a universal way, we also have
\begin{align*}
    \omega(x)\mathcal{M}^+(D^2\varphi)(x) \leq C(d,\lambda,\Lambda,\|\omega\|_{L^\infty}) \quad \mbox{ for } |x|\leq 1/4.
\end{align*}

\end{proof}

\begin{Lemma}\label{Lemma_first_ite_prime}
    There exist universal constants $\varepsilon_0>0$, $0<\mu<1$ and $M>1$ such that if $u\in \overline{S}(|f|)$ in $Q_{4\sqrt{d}}$, $u\in C(\overline{Q}_{4\sqrt{d}})$ and
    \begin{align*}
        &u\geq 0 \mbox{ in } Q_{4\sqrt{d}},\\
        &\inf_{Q_3}u\leq 1,\\
        &\|f\omega^{-1}\|_{ L^d(Q_{4\sqrt{d}})}\leq \varepsilon_0,
    \end{align*}
    then
    \begin{align*}
        |\{u\leq M\}\cap Q_1|>\mu.
    \end{align*}
\end{Lemma}
\begin{proof}
Let $\varphi$ be the barrier function defined in Lemma \ref{Lemma_barrier} and define $v=u+\varphi\in \overline{S}(|f|+C\xi)$. We get then for $v$ by Proposition \ref{Proposition_w_weighted_ABP} we get
\begin{align*}
    1\leq & C\left(\int_{B_{2\sqrt{d}}\cap \{v=\Gamma_v\}}\left(|f|+C\xi\omega\right)^d\omega^{-d}\,dx\right)^\frac{1}{d}\\
    \leq &C\|f\omega^{-1}\|_{ L^d(B_{4\sqrt{d}})}+C|\{v=\Gamma_v\}\cap Q_1|^\frac{1}{d}.
\end{align*}
For $\varepsilon_0$ small enough, we get
\begin{align*}
    |\{u\leq M\}\cap Q_1|^\frac{1}{d}\geq C
\end{align*}
as intended.

\end{proof}
This result is identical to \cite[Lemma 4.5]{Caffarelli-Cabre_1995} and therefore the whole argument in Chapter 4 therein can be repeated. In particular, the proof of Harnack inequality, interior  H\"older continuity and uniform continuity up to the boundary follow identically. We highlight these two latter results.

\begin{Proposition}[Interior H\"older continuity]\label{Proposition_holder_continuity}
Suppose \ref{Assumption1}-\ref{Assumption2} are in force and let $u\in S^*(f)$ in $B_1$. Then $u\in C^{\alpha}(B_{1/2})$ and there exists $C>0$ 
\begin{align*}
    \|u\|_{C^\alpha(B_{1/2})}\leq C\left(\|u\|_{L^\infty(B_1)}+\|f\omega^{-1}\|_{L^d(B_1)} \right),
\end{align*}
where $C$ and $\alpha$ depend on $\lambda$, $\Lambda$, $d$ and $\omega$.
\end{Proposition}

\begin{Proposition}[Uniform continuity up to the boundary]\label{Proposition_continuity_boundary}
    Let $u\in S(f)$ in $B_1$, $\varphi:=u_{|\partial B_1}$ and $\rho_\varphi(|x-y|)$ be a modulus of continuity of $\varphi$. Assume finally that $K>0$ is a constant such that $\|\varphi\|_{L^\infty(\partial B_1)}\leq K$ and $\|f\omega^{-1}\|_{L^d(B_1)}\leq K$.

    Then there exists a modulus of continuity $\rho_u$ of $u$ in $\overline{B_1}$ which depends only on $d, \lambda, \Lambda, K$ and $\rho_\varphi$.
\end{Proposition}

We conclude this section with an argument to flatten the smooth surface $\Gamma=\{\omega=0\}$ and a reduction to the unitary ball. For the rest of the paper, we will always assume to be in this simplified case.

\begin{Remark}\label{Rmk:flatenning}
Let $\omega(x)=|x_d-\Psi(x')|^a$ so that
\[
\Gamma=\{\omega(x)=0\}=\{x_d=\Psi(x')\}
\]
with $\Psi\in C^\infty(B_1')$. Assume, without loss of generality, that $\Psi(0)=0$ and $D_{x'}\Psi(0)=0$. Then the normal of $\Gamma$ at $0$ is $e_d$. Since $\mu^{-a} \omega(\mu x)\to |x_d|^a$, we get by \ref{Assumption1} and \ref{Assumption4} that
\begin{align*}
F_\mu(M,x):=\mu^{-a}F(M,\mu x),
\end{align*}
converges locally uniformly, up to a subsequence, to $\omega(x)\overline{F}(M)$ , where $\overline{F}$ is $(\lambda,\Lambda)$-elliptic. From now on, in view of the previous reduction , we assume that $\Omega= B_1$, $\Gamma = \{x_d=0\}$ and thus $\omega(x)=|x_d|^a$. By standard covering and flattening arguments, we can recover the general case.
\end{Remark}

\section{Uniqueness of $L^p$ viscosity solutions}\label{Section_unique}

In this section, we will prove the crucial uniqueness result of $L^p$ viscosity solutions to the following homogeneous Dirichlet problem
\begin{align*}
    \begin{cases}
        |x_d|^aF(D^2u)=0, &\mbox{ in } B_1,\\
        u=\varphi, &\mbox{ on } \partial B_1.
    \end{cases}
\end{align*}

We use Jensen's approximate solutions, noting that the natural attempt to instead consider $ |x_d^2 +\varepsilon^2|^{\frac{a}{2}}F(D^2u_\varepsilon)=0$ combined with stability results produces a solution to our equation but does not provide uniqueness. 

We start by defining Jensen's approximate solutions. Let $u$ be a continuous function in $B_1$ and let $U$ be an open set such that $\overline{U}\subset B_1$. We define, for $\varepsilon>0$, the upper $\varepsilon$-envelope of $u$ with respect to $U$:
\begin{align*}
    u^\varepsilon(x_0)=\sup_{x\in \overline{U}}\left\{u(x)+\varepsilon-\frac{1}{\varepsilon}|x-x_0|^2\right\}, \quad \mbox{ for } x_0\in U.
\end{align*}
Similarly we can define the lower envelope $u_\varepsilon$ using convex paraboloids.

The following properties of $u^\varepsilon$ are proven in \cite[Lemma 5.2]{Caffarelli-Cabre_1995}. 

\begin{Lemma}\label{Lemma_properties_envelope}
    Let $x_0, x_1\in U$. Then
    \begin{enumerate}
        \item $\exists x_0^*\in \overline{U}$ such that $u^\varepsilon(x_0)=u(x_0^*)+\varepsilon-|x_0^*-x_0|^2/\varepsilon$.
        \item $u^\varepsilon(x_0)\geq u(x_0)+\varepsilon$.
        \item $|u^\varepsilon(x_0)-u^\varepsilon(x_1)|\leq (3/\varepsilon)\mbox{diam}(U)|x_0-x_1|$.
        \item $0<\varepsilon<\varepsilon'\implies u^\varepsilon(x_0)\leq u^{\varepsilon'}(x_0)$.
        \item $|x_0^*-x_0|^2\leq \varepsilon \osc_U u$.
        \item $0<u^\varepsilon(x_0)-u(x_0)\leq u(x_0^*)-u(x_0)+\varepsilon$.
    \end{enumerate}
\end{Lemma}

We will now prove the following.
\begin{Theorem}\label{Theorem_properties_envelope}
    \begin{enumerate}
        \item $u^\varepsilon\in C(U)$ and $u^\varepsilon\downarrow u$ uniformly in $U$ as $\varepsilon\to 0$.
        \item For any $x_0\in U$, there is a concave paraboloid of opening $2/\varepsilon$ that touches $u^\varepsilon$ by below at $x_0$ in $U$. Hence $u^\varepsilon$ is $C^{1,1}$ from below in $U$. In particular, $u^\varepsilon$ is punctually second order differentiable at almost every point of $U$.
        \item Suppose that $u$ is an $L^p$ viscosity solution of $|x_d|^aF(D^2u)=0$ in $B_1$ and that $U_1$ is an open set such that $\overline{U}_1\subset U$. We then have that for $\varepsilon\leq \varepsilon_0$ (where $\varepsilon_0$ depends only on $u, U, U_1$), $u^\varepsilon$ is an $L^p$ viscosity subsolution of $|x_d|^aF(D^2u)=0$ in $U_1$; in particular, $|x_d|^aF(D^2u^\varepsilon(x))\geq 0$ a.e. in $U$.
        \end{enumerate}
\end{Theorem}
\begin{proof}
    The proof of \textit{(1)} and \textit{(2)} can be found in \cite[Theorem 5.1]{Caffarelli-Cabre_1995}. We will prove \textit{(3)}.

    Let $x_0\in U_1$. We want to prove that if $\varphi\in W^{2,p}$ touches $u^\varepsilon$ by above at $x_0$, then
    \begin{align}\label{eq_env_supsol}
        \esslimsup_{x\to x_0}\left(|x_d|^aF(D^2\varphi(x))\right)\geq 0.
    \end{align}
Consider
    \begin{align*}
        \psi(x)=\varphi(x+x_0-x_0^*)+\frac{1}{\varepsilon}|x_0-x_0^*|^2-\varepsilon.
    \end{align*}
    
    By property \textit{(5)} of Lemma \ref{Lemma_properties_envelope}, we can pick $\varepsilon_0$ so small that for every $\varepsilon\leq \varepsilon_0$ we have $x_0\in U_1$ implies $x_0^*\in U$. Hence
\begin{align*}
    u(x)\leq u^\varepsilon(x+x_0-x_0^*)+\frac{1}{\varepsilon}|x_0-x_0^*|^2-\varepsilon.
\end{align*}
Therefore, again for $x$ close to $x_0^*$,
\begin{align*}
    u(x)\leq\,& \varphi(x+x_0-x_0^*)+\frac{1}{\varepsilon}|x_0-x_0^*|^2-\varepsilon\\
    =\,&\psi(x)
\end{align*}
and $u(x_0^*)=\psi(x_0^*)$, since $\varphi(x_0)=u^\varepsilon(x_0)$. Hence $\psi$ touches $u$ by above at $x_0^*$ and therefore
\begin{align*}
    0\leq&\, \esslimsup_{x\to x_0^*}\left(|x_d|^aF(D^2\psi(x))\right)\\
    =&\,\esslimsup_{y\to x_0}\left(|(y-x_0+x_0^*)_d|^aF(D^2\varphi(y))\right).
\end{align*}
Therefore, this implies that
\begin{align*}
\esslimsup_{y\to x_0}\left(F(D^2\varphi(y))\right)\geq 0
\end{align*}
from which \eqref{eq_env_supsol} follows immediately.
    
\end{proof}

The following stability result will be used several times.

\begin{Proposition}\label{Proposition_stability}
    Let $F$ be a $(\lambda,\Lambda)$-elliptic operator and $\{u_k\}_k$ be $L^p$ viscosity supersolutions of $|x_d|^aF(D^2u_k)= 0 $ in $B_1$. Assume also that $u_k$ converges locally uniformly to $u$. Then $u$ is an $L^p$ viscosity supersolution to
    \begin{align*}
        |x_d|^aF(D^2u)\geq 0, \quad\mbox{ in } B_1.
    \end{align*}
\end{Proposition}
\begin{proof}
    Let $\varphi\in W^{2,p}$ touch $u$ from above at $x_0$, $B_r(x_0)\subset B_1$, and $k_0\geq 1$. Then
\begin{align*}
    \varphi_k(x):=\varphi(x)+\frac{1}{k}\frac{|x-x_0|^2}{2}+C_k
\end{align*}
touches $u_k$ from above at some point $x_k\in B_r(x_0)$, for some $k\geq k_0$. Hence
\begin{align*}
    \esslimsup_{x\to x_k}\left(|x_d|^aF\left(D^2\varphi(x)+\frac{1}{k} I\right)\right)\geq 0,
\end{align*}
which implies, by uniform ellipticity of $F$,
\begin{align*}
    \esslimsup_{x\to x_k}\left(\Lambda |x_d|^a\frac{1}{k}+|x_d|^aF(D^2\varphi(x))\right)\geq 0.
\end{align*}
Letting $k\to \infty$ we get $x_k\to x_0$, and thus
\begin{align*}
    \esslimsup_{x\to x_0}\left(|x_d|^aF(D^2\varphi(x))\right)\geq 0.
\end{align*}
as intended.
\end{proof}

\begin{Remark}\label{Remark_closed_S}
    In particular, this implies that the classes $\overline{S}, \underline{S}$ and $S$ are closed under uniform limits in compact sets.
\end{Remark}

\begin{Theorem}\label{Theorem_uniqueness}
    Let $u$ be an $L^p$ viscosity subsolution of $|x_d|^aF(D^2u)=0$ and $v$ be an $L^p$ viscosity supersolution of $|x_d|^aF(D^2v)=0$ in $B_1$. Then
    \begin{align*}
        u-v\in \underline{S}, \mbox{ in } B_1.
    \end{align*}
\end{Theorem}
\begin{proof}

    Fix $U$ and $U_1$ such that $\overline{U}_1\subset U\subset \overline{U}\subset B_1$; we will prove that for $\varepsilon>0$ small enough $u^\varepsilon-v_\varepsilon\in \underline{S}$ in $U_1$. It follows that $u-v\in \underline{S}$ in $B_1$ since $U_1$ was arbitrary, $u^\varepsilon-v_\varepsilon$ converges uniformly to $u-v$ and $\underline{S}$ is closed, see Remark \ref{Remark_closed_S}.

    To see that $u^\varepsilon-v_\varepsilon\in \underline{S}$ in $U_1$, let $\varphi\in W^{2,p}$ touch $u^\varepsilon-v_\varepsilon$ from above at $x_0$ in $B_r(x_0)\subset U_1$. It suffices to show that 
    \begin{align*}
        \esslimsup_{x\to x_0} |x_d|^a\mathcal{M}^+(D^2\varphi(x))\geq 0.
    \end{align*}
    Take $r>0$ so small that $\overline{B}_{2r}(x_0)\subset U$, $\delta>0$ and define
    \begin{align*}
        \psi(x)=v_\varepsilon(x)-u^\varepsilon(x)+\varphi(x)+\delta|x-x_0|^2-\delta r^2.
    \end{align*}
    We have that $\psi\geq 0$ in $\partial B_r(x_0)$ and $\psi(x_0)<0$. Using \textit{(2)} in Theorem \ref{Theorem_properties_envelope}, we know that $\psi\in W^{2,p}$ and therefore it is an $L^p$ strong supersolution of the uniformly elliptic equation
    \begin{align*}
    \mathcal{M}^+(D^2\psi(x))\leq d\Lambda |D^2\psi(x)|\in L^p.
    \end{align*}
    Therefore we can use the ABP obtained in \cite[Proposition 3.3]{Caffarelli-Crandall-Kocan-Swiech_1996} together with the fact that $L^p$ strong solutions are also $L^p$ viscosity solutions to get
    \begin{align}\label{Equation_unique1}
        0< \sup_{B_r(x_0)} \psi^-\leq C\left(\int_{B_r\cap \{\psi=\Gamma_\psi\}}|D^2\psi|^d \, dx\right)^\frac{1}{d}.
    \end{align}

    By \textit{(2)} in Theorem \ref{Theorem_properties_envelope} we know that there exists a set $A$ such that $|B_r(x_0)\setminus A|=0$ and $u^\varepsilon$ and $v_\varepsilon$ (and hence $\psi$) are punctually second order differentiable in $A$. By \textit{(3)} in Theorem \ref{Theorem_properties_envelope},
    \begin{align}\label{Equation_unique2}
        |x_d|^aF(D^2v_\varepsilon(x))\leq 0 \mbox{ and } |x_d|^aF(D^2u^\varepsilon(x))\geq 0,\mbox{ for } x\in A\cap B_r(x_0).
    \end{align}
    Since $\Gamma_\psi$ is convex and $\Gamma_\psi\leq \psi$, we have that
    \begin{align}\label{Equation_unique3}
        D^2\psi(x) \mbox{ is nonnegative definite, for } x\in A\cap B_r(x_0)\cap \{\psi=\Gamma_\psi\}.
    \end{align}
    It follows from \eqref{Equation_unique1} and $|B_r(x_0)\setminus A|=0$ that, for every $r>0$,
    \begin{align*}
        |\{\psi=\Gamma_\psi\}\cap A\cap B_r(x_0)|>0
    \end{align*}
    and that for every $x^r\in \{\psi=\Gamma_\psi\}\cap A\cap B_r(x_0)$, we have by \eqref{Equation_unique2} and \eqref{Equation_unique3} that
    \begin{align*}
        0\leq & \,|x^r_d|^aF(D^2u^\varepsilon(x^r))\\
        = & \,|x^r_d|^aF(D^2v_\varepsilon(x^r)-D^2\psi(x^r)+D^2\varphi(x^r)+2\delta I)\\
        \leq &\,|x^r_d|^aF(D^2v_\varepsilon(x^r)+D^2\varphi(x^r)+2\delta I)\\
        \leq &\, |x^r_d|^a\left[ F(D^2v_\varepsilon(x^r))+\Lambda \|(D^2\varphi(x^r))^+\|-\lambda\|(D^2\varphi(x^r))^-\|+2\Lambda \delta \right]\\
        \leq &\,|x^r_d|^a\left[ \Lambda \|(D^2\varphi(x^r))^+\|-\lambda\|(D^2\varphi(x^r))^-\|+2\Lambda \delta \right]\\
        \leq&\, |x^r_d|^a\mathcal{M}^+(D^2\varphi(x^r))+2\Lambda \delta |x^r_d|^a.
    \end{align*}
    Letting $\delta \to 0$ we get that
    \begin{align*}
        |x^r_d|^a\mathcal{M}^+(D^2\varphi(x^r))\geq 0
    \end{align*}
	holds for every $x^r\in \{\psi=\Gamma_\psi\}\cap A\cap B_r(x_0)$. Since this set has positive measure we conclude that
	\begin{align*}
        \esslimsup_{x\to x_r}|x_d| ^a \mathcal{M}^+(D^2\varphi(x))\geq 0.
    \end{align*}  
    Letting finally $r\to 0$, $x^r\to x_0$ and we conclude
    \begin{align*}
        \esslimsup_{x\to x_0}|x_d|^a\mathcal{M}^+(D^2\varphi(x))\geq 0,
    \end{align*}
    as intended.

\end{proof}

Combining Theorem \ref{Theorem_uniqueness} with the maximum principle, which follows from Proposition \ref{Proposition_w_weighted_ABP}, we immediately get the following uniqueness result. We emphasize that uniqueness only holds for the class of $L^p$ viscosity solutions, since the maximum principle fails for $C$ viscosity solutions.
\begin{Corollary}\label{Corollary_uniqueness}
    The Dirichlet problem
    \begin{align*}
        \begin{cases}
        |x_d|^aF(D^2u)=0,\quad &\mbox{ in } B_1,\\
        u=\varphi, &\mbox{ on } \partial B_1,
    \end{cases}
    \end{align*}
    has at most one $L^p$ viscosity solution $u\in C(B_1)$.
\end{Corollary}

\section{Gradient Regularity}\label{Section_C1alpha}

This section concerns with the higher order regularity for the equation 
\begin{align}\label{eq_var_coef}
    \begin{cases}
        F(D^2u,x)=f, &\mbox{ in } B_1,\\
        u=\varphi, &\mbox{ on } \partial B_1,
    \end{cases}
\end{align}
where $\varphi$ has some modulus of continuity.  We obtain $C^{1,\alpha}$ regularity up to and across the degeneracy set $\Gamma=\{x_d=0\}$.

Our argument can be divided in three steps, first we obtain regularity in $\Gamma=B_1':=B_1\cap \{x_d=0\}$, then we obtain interior regularity (away from $\Gamma$)  and finally we get regularity across $\Gamma$. In more detail,
\begin{enumerate}
\item We start by obtaining pointwise $C^{1,\alpha}$ regularity at the points $x_0\in B_{1/2}'$. We do this by proving that in a small ball $B_r(x_0)$, the graph of $u$ is trapped between two functions $\ell(x)\pm C|x-x_0|^{1+\alpha}$ from both sides of $B_r'$, where $\ell$ is an affine function;

\item Then we obtain interior regularity at points away from $B_1'$. This result follows immediately from \cite[Chapter 8.2]{Caffarelli-Cabre_1995} after a rescaling argument.

\item Combining the two sided control of $u$ on $B_1'$, which follows from step (1), and the interior regularity obtained in step (2) we get regularity up to the boundary immediately from \cite[Proposition 2.2]{Milakis-Silvestre_2006}.
\end{enumerate}

 The approach used in this paper is to consider a zooming on the coefficients and use the scaling given by the degeneracy, i.e. considering
\begin{align*}
F_\mu(M,x)=\mu^{-a} F(M,\mu x).
\end{align*}
The observation in Remark \ref{Rmk:flatenning} to proceed with the argument was inspired by  \cite[Assumption (H2)]{LST} and states that $F_\mu(M,x)$ converges to $|x_d|^a\overline{F}(M)$ uniformly in $x$ and locally in $M$, possibly up to a subsequence.
\begin{Remark}

It is noteworthy that the usual approach to treat the case with variable coefficients is to consider a smallness assumption on the oscillation of the form
\begin{align}\label{eq_intro_beta}
\|\beta\|_{L^\infty(B_r)}\leq Cr^\gamma,
\end{align}
where $\beta$ measures the proximity between the operator $F$ and another one for which there is improved regularity. In \cite[Chapter 8]{Caffarelli-Cabre_1995}, for example, $\beta$ measures the proximity between $F(M,x)$ and $F(M,0)$. In our case, however, since the ellipticity vanishes at points where $x_d=0$, we instead measure the proximity to the operator $|x_d|^a\overline{F}(M)$ for which there is optimal $C^{1,\alpha_0}$ regularity, see Lemma \ref{Lemma_division}. Therefore, it seems natural to consider instead
\begin{align*}
\beta(x)=\sup_{M\in \text{Sym}(d)}\frac{|F(M,x)-|x_d|^a\overline{F}(M)|}{1+|M|}.
\end{align*}
And indeed the argument in Remark \ref{Rmk:flatenning} follows from \eqref{eq_intro_beta}. 

\end{Remark}

\bigskip

\subsection{Regularity on $B_1'$}

We start by considering the more challenging case and obtain $C^{1,\alpha}(x_0)$ regularity at points where $(x_0)_ d=0$. We assume, without loss of generality, that $x_0=0$. Remark \ref{Rmk:flatenning} provides a tangential path from equation \eqref{eq_var_coef} to the following one
\begin{align}\label{eq_div_1}
    |x_d|^a\overline{F}(D^2u)=0, \quad \mbox{ in } B_1,
\end{align}
where $\overline{F}$ is uniformly elliptic with constants $\lambda, \Lambda$.

First we obtain interior $C^{1,\alpha_0}$ regularity for \eqref{eq_div_1} by the following Division Lemma, and then we use a perturbation argument to get $C^{1,\alpha}(x_0)$ for \eqref{eq_var_coef}.

\begin{Lemma}[Division Lemma]\label{Lemma_division}
If $u$ is a viscosity solution to \eqref{eq_div_1}, then it is also a viscosity solution to
\begin{align}\label{eq_div_2}
    \overline{F}(D^2u)=0, \quad \mbox{ in } B_1.
\end{align}
\end{Lemma}
\begin{proof}

Let $u$ be a solution to \eqref{eq_div_1} in $B_1$ and $v$ be the unique solution of
\begin{align*}
\begin{cases}
\overline{F}(D^2v)=0, &\mbox{ in } B_{1},\\
v=\varphi,\qquad &\mbox{ on } \partial B_{1}.
\end{cases}
\end{align*}
Then clearly
\begin{align}\label{eq_div_3}
\begin{cases}
|x_d|^a\overline{F}(D^2v)=0, &\mbox{ in } B_{1},\\
v=\varphi,\qquad &\mbox{ on } \partial B_{1}.
\end{cases}
\end{align}
By Corollary \ref{Corollary_uniqueness}, there exists a unique solution to \eqref{eq_div_3} and thus $v=u$ in $B_1$.

\end{proof}

The following Approximation Lemma is instrumental in our study.
\begin{Lemma}[Approximation Lemma]\label{Lemma_approximation}
Let $u\in C(B_1)$ be a solution to
\begin{align*}
F_\mu(D^2u,x)=f.
\end{align*}
Then, for every $\varepsilon>0$ there exists $\delta>0$ such that if $\mu\leq \delta$, $\|u\|_{L^\infty(B_1)}\leq 1$ and $\|f\|_{L^\infty(B_1)}\leq \delta$, there exists $h\in C^{1,\alpha_0}_{loc}(B_{9/10})$ solving
\begin{align*}
\overline{F}(D^2h)=0,\quad \mbox{ in } B_{9/10}
\end{align*} 
and such that
\begin{align*}
\|u-h\|_{L^\infty(B_{1/2})}\leq \varepsilon.
\end{align*}
\end{Lemma}

\begin{proof}
By contradiction, assume that there exist sequences $(u_n)_n$, $(\mu_n)_n$ and $(f_n)_n$ such that
\begin{align*}
F_{\mu_n}(D^2u_n,x)=f_n
\end{align*}
with $\mu_n\to 0$, $\|u_n\|_{L^\infty(B_1)}\leq 1$ and $f_n\to 0$; however there exists $\varepsilon_0>0$ such that, for every $h\in C_{loc}^{1,\alpha_0}(B_{9/10})$,
\begin{align}\label{eq_approx_1}
\|u_n-h\|_{L^\infty(B_{1/2})}>\varepsilon_0.
\end{align}

By Remark \ref{Rmk:flatenning}, up to a subsequence, $F_{\mu_n}(M,x)\to |x_d|^a\overline{F}(M)$. Since $F_{\mu_n}(M,x)$ has ellipticity  $|x_d|^a\lambda$ and $|x_d|^a\Lambda$, by Proposition \ref{Proposition_holder_continuity}, $u_n\in C_{loc}^\alpha$ with universal estimates. Hence,  by Arzelà-Ascoli, there exists a subsequence such that $u_n\to u_0$ locally uniformly. Arguing as in Proposition \ref{Proposition_stability} we conclude that
\begin{align*}
|x_d|^a\overline{F}(D^2u_0)=0, \quad \mbox{ in } B_{1/2}.
\end{align*}
By Lemma \ref{Lemma_division} we get
\begin{align*}
\overline{F}(D^2u_0)=0, \quad \mbox{ in } B_{1/2}.
\end{align*}
By \cite[Corollary 5.7]{Caffarelli-Cabre_1995}, there exists a universal $\alpha_0$ depending only on $d,\lambda,\Lambda$ such that $u_0\in C^{1,\alpha_0}_{loc}(B_{9/10})$ with universal estimates. Choosing $h=u_0$ in \eqref{eq_approx_1} we get a contradiction, which concludes the proof.

\end{proof}

We now check that we can assume, without loss of generality, that the smallness regime considered in Lemma \ref{Lemma_approximation} holds. For this purpose, let $\delta>0$ to be chosen later. Define the rescaling
\begin{align*}
\tilde{u}(y)=\frac{r_1^{-2}u(r_1y)}{K}
\end{align*}
with
\begin{align*}
K:=r_1^{-2}\|u\|_{L^\infty(B_1)}+\delta^{-1}r_1^{-a}\|f\|_{L^\infty(B_1)}.
\end{align*}
Then $\tilde{u}$ solves
\begin{align*}
\frac{1}{r_1^a K}F(KD^2\tilde{u},r_1y)=\frac{f(r_1y)}{Kr_1^a}
\end{align*}
which we rewrite as
\begin{align*}
\tilde{F}_{r_1}(D^2\tilde{u},y)=\tilde{f}(y).
\end{align*}
Choosing $r_1\leq \delta$, we see that we are in the smallness regime of Lemma \ref{Lemma_approximation}.

We now use the characterization of H\"older spaces to obtain the first instance of a geometric iteration.

\begin{Lemma}\label{Lemma_first_it1}
Let $u\in C(B_1)$ be a solution to \eqref{eq_var_coef}. Then there exist $0<\rho\ll1$ and an affine function
\begin{align*}
\ell(x)=a+b\cdot x
\end{align*}
with $a\in \Rr$ and $b\in \Rr^d$ such that
\begin{align*}
\|u-\ell\|_{L^\infty(B_\rho)}\leq \rho^{1+\alpha}
\end{align*}
where $\alpha$ is any positive number less than $\alpha_0$ and $\rho$ depends only on universal constants and $\alpha$.
\end{Lemma}
\begin{proof}
Take $\varepsilon_0>0$ to be fixed. Note that this choice fixes $\delta>0$ via Lemma \ref{Lemma_approximation}, such that if the smallness assumptions are satisfied, then there exists $h\in C^{1,\alpha_0}$ with universal estimates such that
\begin{align*}
\|u-h\|_{L^\infty(B_{1/2})}\leq \varepsilon_0.
\end{align*}
Let 
\begin{align*}
\ell(x)=h(0)+Dh(0)\cdot x
\end{align*}
and compute
\begin{align*}
\sup_{B_\rho}|u(x)-\ell(x)|\leq&\sup_{B_\rho}|h(x)-\ell(x)|+\sup_{B_\rho}|u(x)-h(x)|\\
\leq &C\rho^{1+\alpha_0}+\varepsilon_0,
\end{align*}
where $C>0$ and $\alpha_0\in(0,1)$ are universal constants. We now make the universal choices
\begin{align*}
\rho=\left(\frac{1}{2C}\right)^\frac{1}{\alpha_0-\alpha}, \quad \varepsilon_0=\frac{1}{2}\rho^{1+\alpha}
\end{align*}
for $\alpha\in(0,\alpha_0)$ to be chosen later. This concludes the proof.
\end{proof}

The next result extends the oscillation control
from Lemma \ref{Lemma_first_it1} to discrete scales $\rho^n$, for every $n\in\Nn$. Furthermore, it
provides a control on the coefficients of the approximating polynomials.

\begin{Lemma}\label{Lemma_geometric_iterations1}
Let $u\in C(B_1)$ be a solution to \eqref{eq_var_coef}. Then there exists a sequence of affine functions $(\ell_n)_n$ of the form
\begin{align*}
\ell_n(x)=a_n+b_n\cdot x
\end{align*}
satisfying 
\begin{align}\label{eq_it_1}
	\|u-\ell_n\|_{L^\infty(B_{\rho^n})}\leq \rho^{n(1+\alpha)}
\end{align}
and
\begin{align}\label{eq_it_2}
|a_n-a_{n-1}|+\rho^n|b_n-b_{n-1}|\leq C\rho^{(n-1)(1+\alpha)}
\end{align}
for every $n\in \Nn$, where $C$ depends only on $d, \lambda$ and $\Lambda$, $\alpha=\min\{\alpha_0^-,1-a\}$ and $a$ is given by \ref{Assumption1}.

\end{Lemma}

\begin{proof}
We argue by induction on $n\in \Nn$. For simplicity of presentation, we split the proof into three steps.

\textbf{Step 1 -- } Let $\ell_0\equiv 0$ and set
\begin{align*}
\ell_1(x):=h(0)+Dh(0)\cdot x
\end{align*}
where $h$ is the function approximating $u$ from the proof of Lemma \ref{Lemma_approximation}. Then, owing to this lemma, \eqref{eq_it_1} and \eqref{eq_it_2} are readily verified in the case $n=1$.

\textbf{Step 2 -- } Suppose now the claim has been verified for $n=k$ and consider the case $n=k+1$. Let $v_k$ be defined as
\begin{align*}
v_k(x)=\frac{(u-l_k)}{\rho^{k(1+\alpha)}}\left(\rho^kx\right).
\end{align*}
Clearly, $v_k$ solves, for $x\in B_1$, the equation
\begin{align*}
\rho^{(1-\alpha-a)k}F\left(\rho^{(\alpha-1)k}D^2v_k,\rho^kx\right)=\rho^{(1-\alpha-a)k}f\left(\rho^kx\right).
\end{align*}
Call $F^k(M,x):=\rho^{(1-\alpha)k}F\left(\rho^{(\alpha-1)k}M,x\right)$ and note that $F^k$ still satisfies assumptions \ref{Assumption1}-\ref{Assumption2} with the same constants and one can also check that Remark \ref{Rmk:flatenning} still holds for $F^k$ with $\overline{F}(M)$ replaced with $\rho^{(1-\alpha)k} \overline{F}(\rho^{(\alpha-1)k}M)$, which still has the same ellipticity contants $\lambda, \Lambda$. Hence we can rewrite the previous equation as
\begin{align*}
F^k_{\rho^k}\left(D^2v_k,x\right)=f_k,
\end{align*}
where $\|f_k\|_{L^\infty(B_1)}\leq \delta_0$ by our choice $\alpha\leq 1-a$.

Applying Lemma \ref{Lemma_approximation}, we find $\overline{h}\in C_{loc}^{1,\alpha_0}(B_{9/10})$ such that
\begin{align*}
\|v_k-\overline{h}\|_{L^\infty(B_{1/2})}\leq \varepsilon_0,
\end{align*}
with
\begin{align*}
\rho^{(1-\alpha)k} \overline{F}\left(\rho^{(\alpha-1)k}D^2\overline{h}\right)=0.
\end{align*}
Furthermore, from Lemma \ref{Lemma_first_it1} we get
\begin{align*}
\sup_{B_\rho}|v_k(x)-\overline{h}(0)-D\overline{h}(x)\cdot x|\leq \rho^{1+\alpha}
\end{align*}
for $\alpha=\min\{\alpha_0^-,1-a\}$. Finally, using the definition of $v_k$ gives
\begin{align*}
\sup_{B_{\rho^{k+1}}}|u(x)-\ell_{k+1}(x)|\leq \rho^{(1+\alpha)(k+1)},
\end{align*}
where
\begin{align*}
\ell_{k+1}(x)=\ell_k(x)+\rho^{(1+\alpha)k}\left(\overline{h}(0)+D\overline{h}(0)\cdot \rho^{-k}x\right).
\end{align*}

\textbf{Step 3 -- } To complete the induction step, we must verify that the coefficients in $\ell_{k+1}$ satisfy \eqref{eq_it_2}. Notice that
\begin{align*}
\begin{cases}
a_{k+1}-a_k=\rho^{(1+\alpha)k}\overline{h}(0)\\
b_{k+1}-b_k=\rho^{\alpha k}D\overline{h}(0).
\end{cases}
\end{align*}
Since $\rho^{(1-\alpha)k} \overline{F}\left(\rho^{(\alpha-1)k}D^2\overline{h}\right)=0$, where $\rho^{(1-\alpha)k} \overline{F}\left(\rho^{(\alpha-1)k} \cdot \right)$ has the same ellipticity constants as $\overline{F}$, $\overline{h}$ enjoys the same $C^{1,\alpha_0}$ estimates as $h$, as thus \eqref{eq_it_2} is verified.

\end{proof}

\bigskip
\subsection{Interior regularity}

We now treat interior points, that is, the points where $|(x_0)_d|>0$. We start with the very simple observation that solutions to
\begin{align*}
F(D^2u,x_0)=0,\quad \mbox{ in } B_1
\end{align*}
belong to $C^{1,\alpha_0}(B_{1/2})$, where $\alpha_0=\alpha_0(\lambda, \Lambda, d)$ is the same as in Lemma \ref{Lemma_approximation}, since the operator $G(M):=|(x_0)_d|^{-a}F(M,x_0)$ is uniformly elliptic with constants $\lambda$ and $\Lambda$. 

We then define the rescaling 
\begin{align*}
\tilde{u}(y)=u(r_1y+x_0)-u(x_0)
\end{align*}
with
\begin{align*}
r_1=|(x_0)_d|/2.
\end{align*}
Then $\tilde{u}$ solves
\begin{align*}
r_1^{2-a}F(r_1^{-2}D^2\tilde{u},r_1y+x_0)=r_1^{2-a}f(r_1y+x_0)
\end{align*}
which we rewrite as
\begin{align*}
\tilde{F}(D^2\tilde{u},x)=\tilde{f}(x),\quad \mbox{ in } B_1
\end{align*}
where $\tilde{F}$ is uniformly elliptic since $r_1^{-a} |r_1 y_d + (x_0)_d|^a\in(1,3^a)$.

Combining this observation with \ref{Assumption4}, we can immediately apply \cite[Theorem 8.3]{Caffarelli-Cabre_1995} and conclude that $\tilde{u}\in C^{1,\alpha}(B_{1/2})$ for every $\alpha<\alpha_0$ and satisfies
\begin{align*}
[\tilde{u}]_{C^{1,\alpha}(B_{1/2})}\leq C\left(\|\tilde{u}\|_{L^\infty(B_1)}+\|\tilde{f}\|_{L^\infty(B_1)}\right)
\end{align*}
where $C=C(\alpha,a,\lambda,\Lambda,d)$. Rescalling back we get
\begin{align*}
[u]_{C^{1,\alpha}(B_{r_1/2}(x_0))}\leq C r_1^{-(1+\alpha)}\left(\osc_{B_{r_1(x_0)}} u +r_1^{2-a}\|f\|_{L^\infty(B_1)}\right).
\end{align*}

\bigskip
\subsection{Regularity across $B_1'$}

The main tool used in this section is the following well-known result, which can be found for example in \cite[Proposition 2.3]{Milakis-Silvestre_2006}. It says that if we have interior regularity in $B_{1/2}^+$ and pointwise regularity on $B_{1/2}'$, then regularity up to the boundary holds in $\overline{B_{1/2}^+}$.

\begin{Proposition}
Let $u$ be a continuous function in $\overline{B_1^+}$ that satisfies the following interior $C^{1,\alpha}$ estimate: If $B_r(x_0)\subset B_1^+$ then 
\begin{align*}
[u]_{C^{1,\alpha}(B_{r_1/2}(x_0))}\leq C_1 r_1^{-(1+\alpha)}\left(\osc_{B_{r_1(x_0)}} u +r_1^{2-a}\|f\|_{L^\infty(B_1)}\right).
\end{align*}
Suppose also that $u$ is $C^{1,\alpha}$ at the boundary from one side, that is, if $x_0\in B_1'$ then there exists an affine function $\ell(x)=a+b\cdot (x-x_0)$ such that
\begin{align}\label{eq_uptobd2}
\sup_{x\in B_r^+(x_0)}|u(x)-\ell(x)|\leq C_2 r^{1+\alpha}.
\end{align}
Then $u\in C^{1,\alpha}(\overline{B_{1/2}^+})$ with constant depending only $C_1$ and $C_2$; additionally, we have $\ell(x)=u(x_0)+Du^+(x_0)\cdot(x-x_0)$, where $Du^+(x_0)$ denotes the gradient of $u$ at $x_0$ coming from $B_1^+$.
\end{Proposition}

Applying this result also in the lower half-ball $B_1^-$ we see that $u$ is $C^{1,\alpha}$ up to the boundary $B_1'$ from either side. Since additionally \eqref{eq_uptobd2} holds in $B_1^-$ for the same affine function $\ell$, we get $Du^+(x_0)=b=Du^-(x_0)$, which concludes that $u$ is $C^{1,\alpha}$ across $B_1'$.

\begin{Remark}
We point out that the regularity result above is optimal, as illustrated by the following simple example. Consider $F(M,x)=|x|^a \,\mbox{Tr}\, M$ for $0<a<1$ so that $\alpha_0=1$. Then $u(x)=|x|^{1+\alpha}$ is a solution to
\begin{align*}
|x|^a\Delta u=f, \mbox{ in } B_1,
\end{align*}
for $f=C|x|^{\alpha+a-1}$ which is bounded and H\"older continuous provided $\alpha\geq 1-a$. Therefore, we can not expect solutions to be more regular than $C^{1,\alpha}$ for $\alpha=1-a$. It is important to note that these counter-examples exist only in the inhomogeneous setting, as in the homogeneous one we can divide the equation by the degeneracy, see Lemma \ref{Lemma_division}.

\end{Remark}

\section*{Acknowledgements}
This work was initiated while the first author was visiting the Department of Mathematics at Johns Hopkins University. The first author acknowledges the department for their hospitality. The second author is partially supported by NSF DMS grant $2154219$, " Regularity {\sl vs} singularity formation in elliptic and parabolic equations". We would like to thank Jonah Duncan and Hongjie Dong  for very thoughtful comments about the present work.

\bibliographystyle{abbrv}
\bibliography{biblio,biblio2}

\end{document}